\documentclass[12pt]{amsart}
\usepackage{amssymb}

\usepackage{stmaryrd}
\usepackage{mathabx}
\usepackage{mathptmx}
\usepackage{mathtools}
\usepackage{cleveref}
\usepackage[all]{xy}
\usepackage{color}
\usepackage{tikz}
\usepackage{todonotes}
\usepackage{paralist}
\numberwithin{equation}{section}
\usepackage[linesnumbered,ruled,vlined]{algorithm2e}
\usepackage{chngcntr}
\usepackage{subcaption}
\usepackage{comment}
\newcommand{\N}{\mathbb{N}}

\newcommand{\Bc}{\mathcal{B}}

\newcommand{\Gc}{\mathcal{G}}

\newcommand{\Nc}{\mathcal{N}}

\DeclareMathOperator{\HVG}{HVG}

\DeclareMathOperator{\nest}{nest}

\DeclareMathOperator{\pnt}{\raise 0.5mm \hbox{\large\bf.}}

\newcommand\thankssymb[1]{\textsuperscript{\@{*}}}

\newcommand{\Bs}{\widetilde{\mathcal{B}}}

\let\phi=\varphi

\newtheorem{theorem}{Theorem}[section]
\newtheorem{lem}[theorem]{Lemma}
\newtheorem{prop}[theorem]{Proposition}
\newtheorem{cor}[theorem]{Corollary}

\newtheorem{lem-def}[theorem]{Lemma and Definition}
\newtheorem{prop-def}[theorem]{Proposition and Definition}

\newtheorem{defi}[theorem]{Definition} 
\newtheorem{rem}[theorem]{Remark}
\newtheorem{rem-def}[theorem]{Remark and Definition}
\newtheorem{exa}[theorem]{Example}

\newtheorem{question}[theorem]{Question}

\begin{document}
\title{Counting Horizontal Visibility Graphs}

\author{Martina Juhnke-Kubitzke}
\address{Universit\"at Osnabr\"uck, Institut f\"ur Mathematik, 49069 Osnabr\"uck, Germany}
\email{juhnke-kubitzke@uos.de}

\author{Daniel K\"ohne} 
\address{Universit\"at Osnabr\"uck, Institut f\"ur Mathematik, 49069 Osnabr\"uck, Germany}
\email{dakoehne@uos.de}

\author{Jonas Schmidt}
\address{Universit\"at Osnabr\"uck, Institut f\"ur Cognitive Science, 49069 Osnabr\"uck, Germany}
\email{jonschmidt@uos.de}

\begin{abstract}
Horizontal visibility graphs (HVGs, for short) are a common tool used in the analysis and classification of time series with applications in many scientific fields. In this article, extending previous work by Lacasa and Luque, we prove that HVGs associated to data sequences without equal entries are completely determined by their ordered degree sequence. Moreover, we show that HVGs for data sequences without and with equal entries are counted by the Catalan numbers and the large Schr\"oder numbers, respectively. 
\end{abstract}

\maketitle


\tableofcontents

\section{Introduction}
\label{sect_intro}
Given a data sequence or a time series it is common to associate a so-called \emph{horizontal visibility graph} to it which can be used for its classification and analysis. In particular, the degree distribution of an HVG is known to be a good measure for distinguishing stochastic from chaotic systems \cite{luque2009horizontal}. HVGs have found applications in many different areas. Besides physics, where they are employed in optics \cite{aragoneses2016}, plasma physics \cite{acosta2021applying} (in a directed version), fluid dynamics \cite{Manshour_2015} or the fault diagnosis of rolling bearings \cite{gao2020fault}, their usage ranges from finance \cite{rong2018topological} to the EEG analysis of epileptics in physiology \cite{physiology_donges2013}, to the identification of alcoholic patients in neuroscience \cite{zhu2014analysis}. In many of those applications, simple metrics such as the vertex degree sequence, the graph entropy and moments have shown to be particularly helpful indicators for the classification of the considered data sequences. It is hence natural to ask whether such properties already determine an HVG. In the case of vertex degree sequences this question is known to have partial answers. In Luque and Lacasa \cite{luque2017canonical} provided an affirmative answer for the class of canonical HVGs by providing an explicit bijection to the set of possible ordered degree sequences. Here, an HVG is called \emph{canonical} if the underlying time series has pairwise distinct entries, and the first and last value are the largest ones. Whereas the first condition is met by most (even discrete) real-world time series with a sufficiently high resolution, the second condition is a huge restriction since it is very unlikely to be satisfied in applications. Another positive result in this direction was provided by O'Pella in \cite{opella2019horizontal} showing that any HVG (without additional requirements) can be recovered from its directed vertex degree sequence and providing an explicit algorithm for this aim. It is essential that the degree sequence is directed since for arbitrary degree sequences it is easy to construct examples where the statement does not hold (see \Cref{figure: exam_same_degree}). However, in all those examples, it turns out that the underlying data sequences have equal entries. Indeed, we prove the following statement:

\begin{theorem}\label{thm: degree sequence determines HVG}
Let $N\in \mathbb{N}$. If $G$ and $H$ are different HVGs on $N$ vertices corresponding to data sequences with pairwise distinct entries, then their ordered vertex degree sequences $\Delta(G)$ and $\Delta(H)$ are different. In particular, HVGs from data sequences without equal entries are uniquely determined by their ordered vertex degree sequence. 
\end{theorem}

This theorem improves the results from \cite{luque2017canonical} and \cite{opella2019horizontal} by weakening the assumptions on the HVG and needing less information to guarantee uniqueness, respectively.
Moreover, we provide an explicit algorithm to reconstruct an HVG from its ordered vertex degree sequence (see \Cref{rem: degree algo}). 

In the second part of this paper, we take a similar viewpoint as in \cite{Mansour} where HVGs are studied from a purely combinatorial perspective and connections with several combinatorial statistics are established. More precisely, we are interested in the number of HVGs on a fixed number of vertices  corresponding to data sequences without and with equal entries, where we consider two HVGs equal if they are equal as labeled graphs. Miraculously, Catalan numbers and large Schr\"oder numbers determine those cardinalities. More precisely, we show the following:

\begin{theorem}\label{thm: Catalan numbers}
Let $N\in \mathbb{N}$ and let $\Gc_{N,\neq}$ and $\Gc_{N}$ be the set of HVGs on $N$ vertices corresponding to data sequences without and with equal entries, respectively. Then:
\begin{enumerate}
\item[(i)] $|\Gc_{N,\neq}|=C_{N-1}$, where $C_{N-1}=\frac{1}{N}\binom{2N-2}{N-1}$ denotes the \emph{$(N-1)$-st Catalan number}.
\item[(ii)] For $N\geq 2$, one has $|\Gc_{N}|=r_{N-2} $, where $r_N$ denotes the \emph{$N$-th large Schr\"oder number}.
\end{enumerate}
\end{theorem}

For \Cref{thm: Catalan numbers} (i) we provide two different proofs: one purely algebraic and one via a bijection to the set of balanced parantheses of length $N-1$. To show \Cref{thm: Catalan numbers} (ii), the main step is to construct a bijection from HVGs on $N$ vertices not containing the edge $1N$ to bracketings of a string of $N-1$ identical letters, which are known to be counted by the $(N-2)$-nd little Schr\"oder number.

The paper is structured as follows. Section 2 provides relevant background on graphs and, in particular  , HVGs, and proves some basic but useful properties of these. In Section 3 we focus on HVGs corresponding to data sequences with pairwise distinct entries. After providing a specific data sequence that realizes a given HVG (see \Cref{thm:dataGraph}) we prove \Cref{thm: degree sequence determines HVG}. In the last two parts of this section, we provide the two different proofs of \Cref{thm: Catalan numbers} (i). In Section 4, we consider arbitrary HVGs. After proving the analogous result to \Cref{thm:dataGraph} (see \Cref{thm: sequence arbitrary data}) in this setting, we turn to the proof of \Cref{thm: Catalan numbers} (ii). We close this article with some open problems and hints to future work in Section 5. 

\section{Preliminaries}
\label{sec: prelim}
In this section, we provide basic background concerning graphs, horizontal visibility graphs and prove some easy properties of the latter that will be useful. For more details on graphs we refer to \cite{Diestel2018Graphs} and for those related to HVGs to \cite{luque2009horizontal}, \cite{Mansour}.

\subsection{Graphs}\label{sect:prel1}
We start by fixing some notation. For $M,N\in\N$ with $M\leq N$, let $[M,N]=\{M,M+1,\ldots,N\}$ and $[N]=\{1,\ldots,N\}$. Given  $G=(V(G),E(G))$ we often write $V$ and $E$ instead of $V(G)$ and $E(G)$, respectively, if it is clear from the context which graph we are referring to. For $v, w\in V(G)$, we use the shorthand notation $vw$ for $\{v, w\}$. If $vw\in E(G)$, $v$ and $w$ are called \emph{neighbors} and the set of all neighbors of $v$ is denoted by $N(v)$. If $V(G)=[N]$, which will be almost always the case, we denote by $m_G(i)$ the \emph{maximal} neighbor of vertex $i$ in $G$, i.e., 
\begin{equation*}
 m_G(i)=\max\{1\leq \ell\leq N~:~i\ell\in E(G)\}.
\end{equation*} 
If $V(G)=[N]$, the sequence $\Delta(G)=(\delta_1,\ldots,\delta_N)$, where $\delta_i=|N(i)|$, is called the  \emph{(ordered) degree sequence} of $G$. We denote by $G\cup e$ and $G\setminus e$ the graph obtained from $G$ by adding and removing an edge $e$, respectively, i.e., $G\cup e= (V,E\cup \{e\})$ and $G\setminus e= (V,E\setminus\{ e\})$. We define deleting a vertex $v$ as $G\setminus \{v\}=(V\setminus \{v\},E\setminus \{vw\ :\ w\in N(v)\})$.  Given $W\subseteq V$, the subgraph \emph{induced} by $W$ is the graph $G_W = (W,\{uv\in E~:~u, v\in W\})$. 

 A graph $G$ on vertex set $[N]$ is called \emph{non-crossing} if there are no vertices $i<j<k<\ell$ with $\{(i,k),(j,\ell)\}\subseteq E(G)$. Intuitively, this means that one can draw the vertices $1,\ldots, N$ on a horizontal line such that all edges are on or above this line and there is no pair of edges that cross. Similarly, a vertex $i\in [N]$ is called \emph{nested} if there exist $1\leq j<i<k\leq N$ such that $jk\in E(G)$. Otherwise, $i$ is called \emph{non-nested}. We use $\Nc(G)$ to denote the set of all non-nested vertices of $G$.

 \subsection{HVGs~--~Horizontal Visibility Graphs}\label{sect:prel2}
 Given $D=(d_1,\ldots,d_N)\in \mathbb{R}^N$, the \emph{horizontal visibility graph} (or \emph{HVG} for short) of $D$ is the graph $\HVG(D)=([N],E)$, where 
  \begin{equation*}
  E=\{ij~:~d_i>d_k<d_j\text{ for all }1\leq i<k<j\leq N\}
  \end{equation*}
  (see \Cref{figure: example_hvg_1} for an example of a data sequence and its corresponding HVG). Since an HVG is clearly invariant under translation of the underlying sequence $D$ by any vector with equal entries, it does not cause any restriction to consider only non-negative data sequences. This also makes sense from the point of view of applications since there $D$ usually is a data sequence or a time series with non-negative entries. It is also motivated by those applications, that HVGs have to be considered as graphs with fixed vertex labels $1,\ldots, N$ and consequently, two HVGs are considered to be the same if and only if their edge sets are the same and not just if they are isomorphic as unlabeled graphs. We set $\Gc_N=\{\HVG(D)~:~D\in \mathbb{R}_{\geq0}^N\}$ and $\Gc_{N,\neq}=\{\HVG(D)~:~D=(d_1,\ldots,d_N)\in \mathbb{R}_{\geq0}^N,\;d_i\neq d_j\text{ for all } 1\leq i<j\leq N\}$. We note that those two sets are different for $N\geq 4$ (see \Cref{exa: easy properties} for an example). Moreover, we use $P_{N}$ to denote the (inclusion)-minimal HVG in both $\Gc_{N}$ and $\Gc_{N,\neq}$, i.e., $P_N=([N],\{i (i+1)~:~1\leq i\leq N-1\})$. 
   
  \begin{figure}[h]
	\centering
	\begin{subfigure}[b]{.45\textwidth}
		\centering
		\includegraphics[width=.4\linewidth]{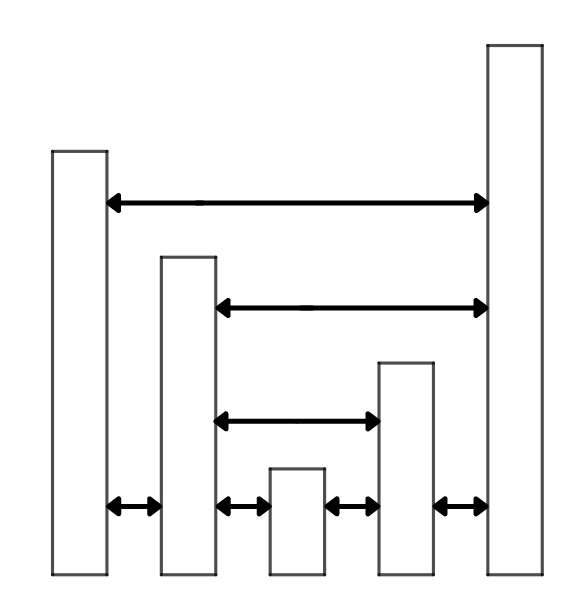}
		\caption{Arrows between entries of $D$ indicate edges in the corresponding HVG.}
		\label{figure: example_hvg_a}
	\end{subfigure}\hspace{5mm}%
	\begin{subfigure}[b]{.45\textwidth}
		\centering
		\includegraphics[width=.6\linewidth]{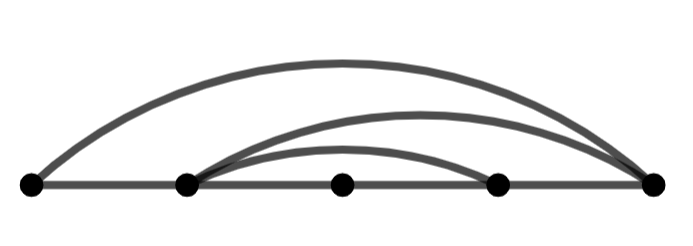}
		\vspace{0.5cm}
		\label{figure: example_hvg_graph_b}
		\caption{The HVG associated to $D=(4, 3, 1, 2, 5)$.\\\ }
	\end{subfigure}
	\caption{The data sequence $D=(4, 3, 1, 2, 5)$ and its associated HVG.}
	\label{figure: example_hvg_1}
\end{figure}
\begin{rem}\label{rem:integral}
Let $D=(d_1,\ldots,d_N)\in \mathbb{R}^N$. We define $\Phi_N:\mathbb{R}^N\to [N]^N$ by
$$
\Phi_N(D)_i=|\{j~:~d_j\leq d_i\}| \quad \text{for } 1\leq i\leq N,
$$
i.e., $\Phi(D)$ reflects the order of the entries of $D$. 
Then, obviously, $\HVG(D)=\HVG(\Phi(D))$ and hence, any HVG is the HVG of a vector of non-negative integers (of size at most $N$). Moreover, if all entries of $D$ are distinct, then $\Phi(D)$ is a permutation of $[N]$. 
\end{rem} 
We summarize some easy but useful property of HVGs in the following lemma. First we have to introduce a simple graph operation. Given $G\in \Gc_N$ and $H\in\Gc_M$, we use $G+H$ to denote the $1$-sum of $G$ and $H$ with respect to the vertices $N\in V(G)$ and $1\in V(H)$, i.e., $G+H$ is obtained by taking the union of $G$ and $H$ and identifying the vertices $N\in V(G)$ and $1\in V(H)$.  To simplify notation, vertices of $V(H)\setminus\{1\}$ will be numbered with $ N+1,\ldots, N+M-1$ in $G+H$ and the identified vertex will be numbered with $N$.
\begin{lem} \label{lem:easy properties}
Let $N\in\mathbb{N}$ and $G\in \Gc_N$. Let $\Nc(G)=\{i_1<\cdots<i_k\}$ and let $\ell\in \Nc(G)$. Then 
	\begin{enumerate}
			\item[(i)] $G$ is non-crossing.
	\item[(ii)] $1$, $N$ and $m_G(\ell)$ are non-nested.
		\item[(iii)] Let $1\leq i<j\leq N$, then (after relabelling the vertices) $G_{[i,j]}\in \Gc_{j-i+1}$. Moreover, if $G\in \Gc_{N,\neq}$, then $G_{[i,j]}\in \Gc_{j-i+1,\neq}$.
			\item[(iv)]  $i_ji_{m}\in E(G)$ if and only if $m=j+1$ or $m=j-1$.
			\item[(v)] There exists $D=(d_1,\ldots,d_N)\in \mathbb{N}^N$ such that $\HVG(D)=G$ and $d_1=N$. Moreover, $D$ can be chosen as a permutation if $G\in \Gc_{N,\neq}$.
			\item[(vi)] $G=G_{[i_1,i_2]}+\cdots +G_{[i_{k-1},i_k]}$.
	\end{enumerate}
\end{lem}
\begin{proof}
(i) was shown in \cite[Corollary 5]{Mansour}.

For (ii) note that $1$ and $N$ are non-nested by definition. Now assume by contradiction that $m_G(\ell)\notin\Nc(G)$. Together with (i) it follows that there exist $i<\ell<m_G(\ell)<m$ with $im\in E(G)$. But then $\ell$ is nested, a contradiction. 

For (iii) it suffices to note that after relabelling the vertices of $G_{[i,j]}$ by $1,\ldots, j-i+1$ increasingly, $G_{[i,j]}=\HVG((d_i,d_{i+1},\ldots,d_j))$, where $G=\HVG((d_1,\ldots,d_N))$. The second statement is now obvious.

For (iv), let $i_m \in \Nc(G)\setminus\{N\}$. (ii) implies that $m_G(i_m)=i_{j}$ for some $j$. Moreover, we must have $j=m+1$ since otherwise $i_{m+1}\notin \Nc(G)$. This shows $i_m i_{m+1}\in E(G)$. The same argument also shows that $i_m i_j\notin E(G)$ if $j\geq m+2$. The claim follows.

(v) This is an easy consequence of \Cref{thm:dataGraph} and \Cref{thm: sequence arbitrary data}.

(vi) follows from  (iii) and (iv).
\end{proof}
We provide an example to illustrate the difference between $\Gc_{N}\setminus\Gc_{N,\neq}$.
\begin{exa}\label{exa: easy properties}
 The graph $G=\HVG((3,1,1,4))$, shown in \Cref{figure: exa_easy properties 1}, is the (inclusionwise) smallest HVG that cannot be realized by a sequence with pairwise distinct entries, i.e., $G\in \Gc_{4}\setminus\Gc_{4,\neq}$. The sequence $(4,1,1,3)$ yields the same HVG and satisfies the assumption from (v) of the previous lemma. 
    \begin{figure}[h]
	\centering
	\includegraphics[width=0.3\linewidth]{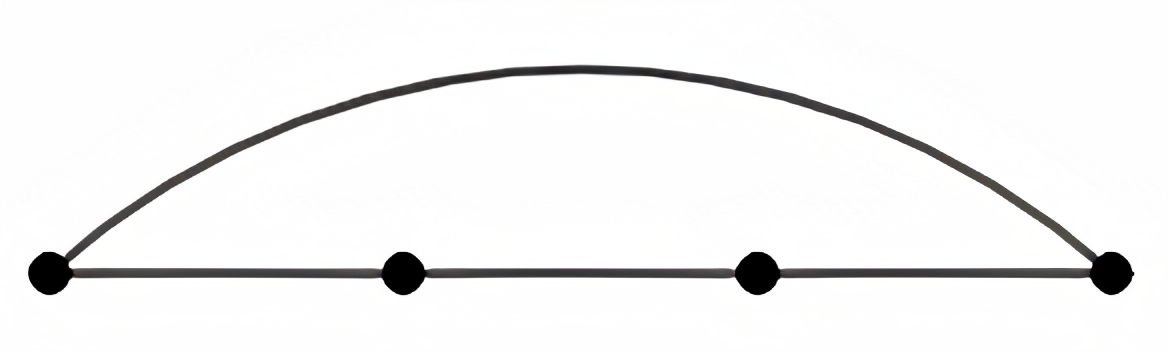}
	\caption{$G=\HVG((4, 1, 1, 3))$.}
	\label{figure: exa_easy properties 1}
\end{figure}
\end{exa}

Motivated by \Cref{lem:easy properties} (iii) it is natural to ask if the set of all HVGs (without fixing the vertex set) is closed under certain graph operations. 

\begin{lem}\label{lem:operations}
Let $M,N\in\mathbb{N}$, $G\in \Gc_N,\ H\in \Gc_M$ and $e\in E(G)\setminus\{i(i+1)~:~1\leq i\leq N-1\}$. Then:
\begin{enumerate}
\item[(i)] $G\setminus e\in \Gc_N$.
\item[(ii)] If $j,\ell\in \Nc(G)$ with $f=j \ell\notin E(G)$, then $G\cup f\in \Gc_N$.
\item[(iii)] $G+H\in \Gc_{N+M-1}$. Moreover, if $G\in\Gc_{N,\neq}$ and $H\in\Gc_{N,\neq}$, then $G+H\in\Gc_{N+M-1,\neq}$.
\end{enumerate}
\end{lem}
Before providing the proof of this lemma, we want to remark that (i) and (ii) are not true if one restricts to $\Gc_{N,\neq}$. For instance, the graph $G$ in \Cref{figure: exa_easy properties 1} is obtained from $G\cup\{13\}\in \Gc_{4,\neq}$ and $P_4\in\Gc_{4,\neq}$ by removing and adding the edge $13$ and $14$, respectively.

\begin{proof}
Let $D=(d_1,\ldots,d_N)\in\mathbb{N}^N$ such that $\HVG(D)=G$. For (i) assume that $e=k\ell$ with $k<\ell$. Let $m=\max\{d_i~:~k<i<\ell\}$ and let $M=\{k<i<\ell~:~d_i=m\}$. Define $\widetilde{D}=(\widetilde{d}_1,\ldots,\widetilde{d}_N)$ by 
\begin{equation*}
\widetilde{d}_i=\begin{cases}
d_i,\quad &\text{if } i\notin M\\
\min(d_k,d_\ell),\quad &\text{if } i\in M.
\end{cases}
\end{equation*}
It is straightforward to show that $\HVG(\widetilde{D})=G\setminus e$, which proves the claim.

For (ii) let $m=\max\{d_i~:~i\in \Nc(G), j\leq i\leq \ell\}$ and set
\begin{equation*}
\widetilde{d}_i=\begin{cases}
m+1,\quad &\text{if } i\in\{j,\ell\}\\
m,\quad &\text{if } i \in \Nc(G) \text{ and } j<i<\ell\\
d_i,\quad \text{otherwise.} 
\end{cases}
\end{equation*}
It is easy to see that $\HVG(\widetilde{D})\supset G\cup f$. If there exists $e\in E(\HVG(\widetilde{D}))\setminus E(G\cup f)$, then we can apply (i) and delete those edges.

For (iii) we can assume by \Cref{lem:easy properties} (v) that $d_1=N$. Let further $F=(f_1,\ldots,f_M)\in \mathbb{N}^M$ with $\HVG(F)=H$ and $f_1=M$. Define $K=(k_1,\ldots,k_{M+N-1})\in\mathbb{N}^{M+N-1}$ by 
\begin{equation*}
k_i=\begin{cases}
d_i+M\quad &\text{if } 1\leq i\leq N\\
f_{i-N+1} \quad &\text{if } N<i\leq M+N-1
\end{cases}
\end{equation*}
and set $J=\HVG(K)$. We obviously have $J_{[N]}=G$ and since $k_N=d_N+M> M=f_1\geq f_{i-N+1}$ for all $N<i\leq M+N-1$ it also holds that $J_{[N,N+M-1]}=H$. The same argument shows that $ij\notin E(J)$ for any $1\leq i< N$ and $N< j\leq M+N-1$, which, together with the previous discussion, implies $J=G+H$ and hence $G+H\in\Gc_{N+M-1}$. Moreover, if $D$ and $F$ have only distinct entries, so does $K$, which shows the second claim.
\end{proof}
\begin{rem}
Combining \Cref{lem:easy properties} (iii) and (vi) with \Cref{lem:operations} (i) and (ii) it is easy to see that if $G\in \Gc_N$ and $1\le j<\ell\le N$ such that $G\cup j\ell$ is non-crossing, then $G\cup j\ell\in\Gc_N.$ 
\end{rem}


\section{Horizontal Visibility Graphs from distinct data}
     \label{sec: Distinct Data}

In this section, we focus on HVGs in $\Gc_{N,\neq}$, i.e., HVGs corresponding to data sequences with pairwise distinct entries. We have seen in \Cref{rem:integral} that such an HVG is the HVG of some permutation of $[N]$. Our first goal is to construct such a permutation, just from the knowledge of the graph, without knowing a realizing data sequence. In the second part of this section, we prove \Cref{thm: degree sequence determines HVG}, i.e., we show that HVGs in $\Gc_{N,\neq}$ are uniquely determined by their \emph{ordered} degree sequence. Our proof also yields an algorithm how to construct such a data sequence. This extends corresponding results for \emph{directed} ordered degree sequences (see \cite[Proposition 5]{opella2019horizontal}) as well as for HVGs in canonical form, i.e., HVGs such that the first and last entry of the corresponding data sequence are the maximal ones \cite[Theorem 1]{luque2017canonical}. In the last two subsections, we consider the enumerative question of how many HVGs in $\Gc_{N,\neq}$ exist. In particular, we provide two proofs of \Cref{thm: Catalan numbers} (i), a purely algebraic one and a bijective one.

\subsection{From HVGs in $\Gc_{N,\neq}$ to data sequences}\label{subsect:GraphToData}
In the following, we let $N\in\mathbb{N}$, $G=([N],E)\in\Gc_{N,\neq}$ and we are seeking  $D\in\mathbb{N}^N$ such that $\HVG(D)=G$. To this end, we first need some further notation. 
A vertex $v\in [N]$ is called \emph{$m$-nested} if
$$
d_{\nest}(v):=|\{ij~:i<v<j, ij\in E\}|=m.
$$
$d_{\nest}(v)$ is also called the \emph{nesting degree} of $v$ and an edge $ij\in E$ with $i<v<j$ is referred to as \emph{nesting edge} of $v$.

\begin{theorem}\label{thm:dataGraph}
    Let $N\in \mathbb{N}$ and $G\in\Gc_{N,\neq}$. Let $\sigma: [N]\to [N]$ be the unique permutation of the vertices of $G$ such that:
    \begin{itemize}
        \item[(i)] $d_{\nest}(\sigma^{-1}(1))\geq d_{\nest}(\sigma^{-1}(2))\geq \cdots\geq d_{\nest}(\sigma^{-1}(N-1))\geq d_{\nest}(\sigma^{-1}(N))$, and,
        \item[(ii)] if $d_{\nest}(\sigma^{-1}(i))=d_{\nest}(\sigma^{-1}(j))$, then $i<j$ iff $\sigma^{-1}(i)>\sigma^{-1}(j)$.
    \end{itemize}
    Let $d_i=\sigma(i)$ and $D=(d_1,\ldots,d_N)$. Then, $D$ realizes $G$, i.e., $HVG(D)=G$. In particular, $d_1=N$.
\end{theorem}
Intuitively, the permutation $\sigma$ corresponds to the ordering $\sigma^{-1}(1),\ldots,\sigma^{-1}(N)$ of the vertices of $G$, that first orders the vertices by decreasing nesting degree and then from right to left among vertices with the same nesting degree. In particular, the vertex $1$ is always the vertex at the last position, i.e., $d_1=\sigma(1)=N$. 
In the following, we refer to the sequence $D$ in \Cref{thm:dataGraph} as the \emph{standard sequence} of a given HVG. 

\begin{proof}
Let $\widetilde{D}=(\widetilde{d}_1,\ldots,\widetilde{d}_N)\in\mathbb{N}^N$ with $\HVG(\widetilde{D})=G$ and let $H=\HVG(D)$. We need to show that $H=G$. For this aim let $ij\in E(G)$ with $2\leq i+1<j\leq N$. Since we must have $\widetilde{d}_i>\widetilde{d}_k<\widetilde{d}_j$ for $i<k<j$, there is no edge in $E(G)$ of the form $uv$ with $u<i<v<j$ or $i<u<j<v$. In particular, if $uv$ is a nesting edge of $i$, then $u<i<j\leq v$ and thus $uv$ is a nesting edge for any $i<k<j$. As $ij$ is also a nesting edge for any $i<k<j$, we conclude that $d_{\nest}(k)>d_{\nest}(i)$, i.e., $d_k=\sigma(k)<\sigma(i)=d_i$ for any $i<k<j$. The analogous reasoning shows $d_k<d_j$ for $i<k<j$. This implies $ij\in E(H)$. 

Let now $ij\in E(H)$ with $2\leq i+1<j\leq N$. In order to show that $ij\in E(G)$ we need to prove that $\widetilde{d}_i>\widetilde{d}_k<\widetilde{d}_j$ for all $i<k<j$. Assume by the contrary that there exists $i<k<j$ with $\widetilde{d}_k> \min(\widetilde{d}_i,\widetilde{d}_j)$. We distinguish two cases.

\noindent{\sf Case 1:} There exists $i<k<j$ with $\widetilde{d}_k>\widetilde{d}_j$. Let $k$ be the maximal vertex with this property. We then have  $\widetilde{d}_\ell<\widetilde{d}_j$ for all $k<\ell<j$. If $uv$ is a nesting edge of $k$ (in $G$), it follows that $u<k<j<v$ and hence $d_{\nest}(j)\geq d_{\nest}(k)$ (in $G$). Using that $k<j$ we infer that $d_j=\sigma(j)<\sigma(k)=d_k$, which contradicts the assumption that $ij\in E(H)$. Thus, $ij\in E(G)$. 

\noindent{\sf Case 2:} There exists $i<k<j$ with $\widetilde{d}_k>\widetilde{d}_i$ and $\widetilde{d}_\ell<\widetilde{d}_j$ for all $i<\ell<j$.
Let $k$ be minimal with this property. Similar arguments as in Case 1 show that $d_{\nest}(k)\geq d_{\nest}(j)$. If $d_{\nest}(k)=d_{\nest}(j)$, we conclude that $d_k=\sigma(k)>\sigma(j)=d_j$ (as $k<j$) which is a contradiction to $ij\in E(H)$. If $d_{\nest}(k)>d_{\nest}(j)$, there has to exist an edge $uv\in E(G)$ with $u<k<v<j$. Since $\widetilde{d}_k>\widetilde{d}_i$ and $ \widetilde{d}_\ell<\widetilde{d}_i$ for all $i<\ell<k$, we must have $u<i$. 
It follows from the first part of this proof that we also have $uv\in E(H)$. But then the edges $uv$ and $ij$ are crossing in $H$, contradicting the fact that $H$ is an HVG (see \Cref{lem:easy properties} (i)). Hence, $ij\in E(G)$. 

Since $i(i+1)$ for $1\leq i\leq N-1$ lies in any HVG, we conclude $G=H$.
\end{proof}

Next, we provide an example for the standard sequence.
\begin{exa}\label{exa: easy properties2}
The sequences $(4,3,1,2,7,5,6)$ and  $(7,4,2,3,6,1,5)$ both realize the HVG shown in \Cref{figure: exa_standardform}. The second sequence meets the condition in \Cref{lem:easy properties} (v) and is constructed using \Cref{thm:dataGraph}.
\begin{figure}[h]
	\centering
	\includegraphics[width=0.7\linewidth]{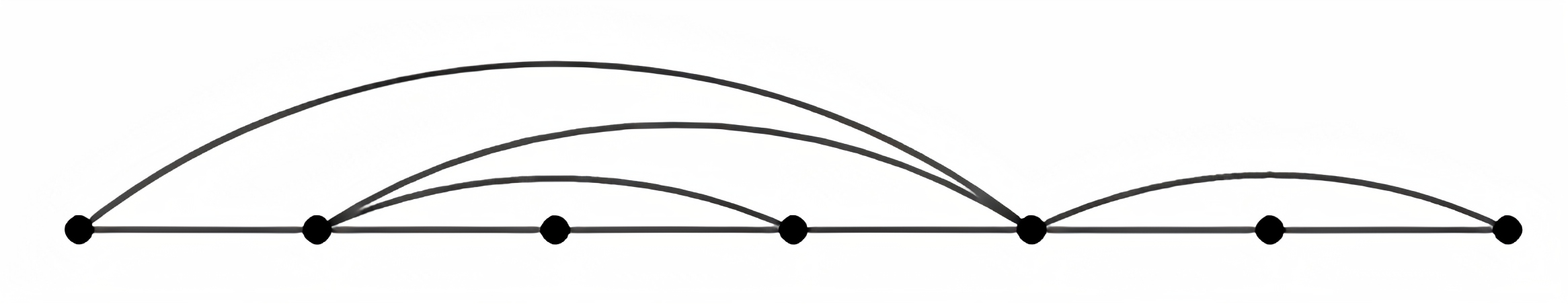}
	\caption{The HVG of the sequences  $(4,3,1,2,7,5,6)$ and  $(7,4,2,3,6,1,5)$.}
	\label{figure: exa_standardform}
\end{figure}

\end{exa}

\subsection{HVGs and degree sequences}\label{sec: Bij thm}

We start with some simple lemmas that will be crucial to prove that an HVG in $\Gc_{N,\neq}$ is uniquely determined by its ordered degree sequence.

\begin{lem}\label{lem: inner 2} 
    Let $N\in \mathbb{N}$, $N\geq 3$ and $G\in\Gc_{N,\neq}\setminus \{P_{N}\}$. Then there exists $2\leq i\leq N-1$ such that $\delta_i=2$ and $(i-1)(i+1)\in E(G)$. 
\end{lem}

\begin{proof}
 Let $D\in [N]^N$ be the standard sequence of $G$ (see \Cref{thm:dataGraph}). We then have $d_1=N$ and $d_N=N-|\Nc(G)|+1$. Since $G\neq P_N$ we have $|\Nc(G)|<N$ and hence $d_N\neq 1$. In particular, there exists $1< i< N$ with $d_i=1$. As $D$ is the standard sequence, we have $d_j>1$ for all $j\neq i$ which implies that $\delta_i=2$ and $(i-1)(i+1)\in E(G)$.
\end{proof}

The drawback of the previous lemma is that we cannot yet tell from a given degree sequence which inner $2$s fulfill the assumption of the corresponding neighboring vertices being adjacent. The next lemma solves this difficulty. 
\begin{lem}\label{lem: first 2}
Let $N\in \mathbb{N}$, $N\geq 3$, $G\in \Gc_{N,\neq}\setminus\{P_N\}$ and $\Delta=(\delta_1,\ldots,\delta_N)$ be the ordered degree sequence of $G$. Then:
\begin{enumerate}
\item[(i)] If $\delta_2=2$ and $\delta_1\neq 1$, then $13\in E(G)$.
\item[(ii)] If $\delta_2\neq 2$ or $\delta_1=1$ and $3\leq i\leq N-1$ is minimal with $\delta_i=2$ and $\delta_{i-1}\geq 3$, then $(i-1)(i+1)\in E(G)$. 
\end{enumerate}
\end{lem}

\begin{proof}
We first note that for $N=3$, $(2,2,2)$ is the only degree sequence meeting the conditions in (i). Since the corresponding HVG is $([3],\{12,13,23\})$, the claim follows in this case.
 
Let $N\geq 4$ and let $D=(d_1,\ldots,d_N)\in \mathbb{N}^N$ be the standard sequence of $G$. 
First assume that we are in situation (i). As $D$ is the standard sequence of $G$, we have $d_1=N>d_2$. If, by contradiction, $13\notin E(G)$, it follows that $d_1>d_2>d_3$. Let $3<m\leq N$ be minimal with $d_m>d_2$. Note that such $m$ exists since $\delta_1\geq 2$ implies the existence of $3<\ell\leq N$ with $1\ell\in E(G)$ and hence $d_\ell>d_2$. It follows that $2m\in E(G)$, a contradiction to $\delta_2=2$.

Now assume that the assumptions of (ii) are satisfied. We first show that there exists $i+1\le\ell\leq N$ with $(i-1)\ell\in E(G)$. This is clear if $i=3$ since $\delta_{2}\geq 3$. So let $i>3$. If there is no such edge, there has to exist an edge $j(i-1)$ with $1\leq j\leq i-3$. \Cref{lem:easy properties} (iii) implies that $G_{[j,i-1]}\in\Gc_{i-j,\neq}\setminus\{P_{i-j}\}$. As $i-j\geq 3$, we conclude with \Cref{lem: inner 2} that there exists an inner vertex $k$ of $G_{[i-j]}$ of degree $2$. In the following we choose $k$ minimal. \Cref{lem:easy properties} (i) together with the fact that $j(i-1)\in E(G)$ implies that $\delta_k=2$ also in $G$. By assumption, we further have $k\neq 2$ and the minimality of $k$ implies $\delta_{k-1}\geq 3$. Since $i$ was the minimal vertex of degree $2$ in $G$, we have hence reached a contradiction. Hence there exists $i+1\le\ell\leq N$ with $(i-1)\ell\in E(G)$. 
If $i=N-1$, we must have $\ell=i+1$ and the claim follows. If $i\neq N-1$, we must have that $d_{i-1}>d_i$. If, by contradiction, $(i-1)(i+1)\notin E(G)$, we conclude that $d_{i-1}>d_i>d_{i+1}$. The claim now follows by the same argument as in (i). 
\end{proof}
We want to point out that \Cref{lem: inner 2} guarantees that the degree sequence of any HVG in $\Gc_{N,\neq}$ either satisfies property (i) or (ii) of the previous lemma or is the one of the trivial HVG. As a consequence, it follows that for any $G\in \Gc_{N,\neq}\setminus\{P_N\}$ there exists $2\leq i\leq N$ with $\delta_i=2$ and $(i-1)(i+1)\in E(G)$. Moreover, the following example shows that it is important to choose $i$ minimally in (ii) since otherwise the statement is not necessarily true. 

\begin{exa} 
The $\HVG$ $G=\HVG(D)$ with $D=(1,8,4,7,6,5,2,3)$ (see \Cref{figure: exa_lem3_3}) has the ordered degree sequence $\Delta=(1,3,2,3,2,3,2,2)$. Vertex $5$ fulfills the assumptions of (ii) except for being minimal and $46\notin E(G)$. However, the minimal inner $2$ is at position $3$ and $24\in E(G)$.
\begin{figure}[h]
    \centering
    \includegraphics[width=10.0cm]{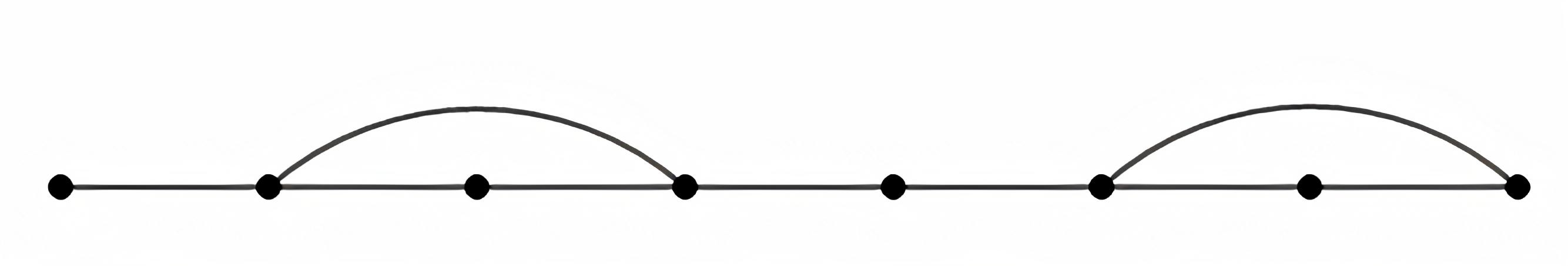}
    \caption{The graph $\HVG((1,8,4,7,6,5,2,3)).$}
    \label{figure: exa_lem3_3}
\end{figure}
\end{exa}

The next lemma shows the behavior of the set of degree sequences of HVGs in $\Gc_{N,\neq}$ with respect to the removal of certain inner $2$s. 
\begin{lem}\label{lem:remove degree 2 vertex}
Let $N\in \mathbb{N}$, $N\geq 3$, $G\in \Gc_{N,\neq}\setminus \{P_N\}$ and $D\in \mathbb{N}^N$ with $\HVG(D)=G$. Let $2\leq i\leq N-1$ with $\delta_i=2$ and $(i-1)(i+1)\in E(G)$ and let $D[i]\in \mathbb{N}^{N-1}$ denote the sequence obtained from $D$ by removing the $i$-th entry. Then $\HVG(D[i])=G\setminus \{i\}$ (after relabelling the vertices $i+1,\ldots,N$ of $G$ by $i,\ldots, N-1$). In particular, $G\setminus \{i\}\in\Gc_{N-1,\neq}$.
\end{lem}
\begin{proof}
As $G\in\Gc_{N,\neq}$ we may assume that all entries of $D$ are distinct. Since $(i-1)(i+1)\in E(G)$, we must have $d_{i-1}>d_i<d_{i+1}$. Let $D[i]\in \mathbb{N}^{N-1}$ be the sequence obtained from $D$ by removing $d_i$ and let $H=\HVG(D[i])$. We claim that $H$ equals $G\setminus \{i\}$ (up to relabelling the vertices of $G\setminus \{i\}$ with $1,\ldots,N-1$). Clearly, $H_{[i-1]}=(G\setminus \{i\})_{[i-1]}$ and $H_{[i,N-1]}=(G\setminus \{i\})_{[i+1,N]}$. Moreover, as $(i-1)~(i+1)\in E(G)$, we also have $H_{[i-1,i]}=(G\setminus \{i\})_{\{i-1,i+1\}}$. Now assume that $1\leq j\leq i-1<\ell\leq N-1$ with $\{j,\ell\}\neq\{i-1,i+1\}$. Then $j\ell\in E(G\setminus \{i\})$ if and only if $d_j>d_k<d_\ell$ for all $j<k<\ell$. As $d_{i-1}>d_i<d_{i+1}$ this is equivalent to $d_j>d_k<d_\ell$ for all $j<k<\ell$ with $k\neq i$, i.e., $j~(\ell-1)\in E(H)$. This completes the proof.  
\end{proof}
We now prove the main result of this subsection, showing that HVGs in $\Gc_{N,\neq}$ are uniquely determined by their ordered degree sequence.\\

\noindent{\sf Proof of \Cref{thm: degree sequence determines HVG} }
We show the claim by induction on $N$. If $N\in\{1,2\}$, then $\Gc_{N,\neq}=\{P_N\}$ and the statement is trivially true. Let $N\geq 3$, $\Delta=(\delta_1,\ldots,\delta_N)\in\mathbb{N}^N$. If $\Delta=(1,2,\ldots,2,1)$, we clearly have $\Delta(P_N)=\Delta$ and as $P_N\subsetneq G$ for any $G\in \Gc_{N,\neq}$, the claim follows in this case. Let $\Delta\neq (1,2,\ldots,2,1)$. Assume there exists $G,H\in \Gc_{N,\neq}$ with $\Delta(G)=\delta=\Delta(H)$. Let $2\leq i\leq N-1$ be minimal such that $\delta_i=2$ and $(i-1)(i+1)\in E(G)\cap E(H)$. Note that such $i$ exists due to \Cref{lem: first 2}.   \Cref{lem:remove degree 2 vertex} implies that $G\setminus \{i\}, H\setminus \{i\}\in\Gc_{N,\neq}$ and, as those graphs have the same degree sequence $\widetilde{\delta}=(\delta_1,\ldots,\delta_{i-2},\delta_{i-1}-1,\delta_{i+1}-1,\delta_{i+2},\ldots,\delta_N)$, the induction hypothesis yields $G\setminus \{i\}=H\setminus \{i\}$. As $i-1$ and $i+1$ are the only neighbors of $i$ in both $G$ and $H$, we conclude that $G=H$. \hfill$\qed$

\begin{rem}\label{rem: degree algo}
The proof of \Cref{thm: degree sequence determines HVG} can easily be turned into an algorithm to construct the unique HVG $G\in \Gc_{N,\neq}$ with a given ordered degree sequence $\Delta$. More precisely, one successively removes the minimal inner $2$ satisfying (i) or (ii) of \Cref{lem: first 2} from $\Delta$ and decreases the two neighboring entries by $1$. We note that for each removal the length of the sequence decreases by $1$. If the $i$-th entry is removed, one protocols the two edges corresponding to the removal (see \Cref{lem:remove degree 2 vertex}). In this way, one finally reaches a sequence of the form $(1,2,\ldots,2,1)$, where also no inner $2$ is possible. If the $2$s have been obtained from the original entries at positions $j_1,\ldots,j_k$, one needs to add the edges $1j_1,j_1j_2,\ldots,j_{k-1}j_k,j_kN$ to the list of edges. We illustrate this procedure in an example.
\end{rem}

\begin{exa}
We consider the ordered degree sequence $\Delta=(2,3,2,5,2,2)$ and construct the corresponding HVG $G$ as follows. We use $E_i$ and $\Delta_i$ to denote the edge set and the changed degree sequence after the removal of $i$ inner $2$s.
\begin{itemize}
    \item We first remove the $2$ at position $3$, which yields $\Delta_1=(2,2,4,2,2)$ and $E_1=\{23,34\}$.
    \item We remove the $2$ at position $2$ and get $\Delta_2=(1,3,2,2)$ and the new edges $12$ and $24$ since what is now vertex $3$ was the original vertex $4$.
    \item We remove the $2$ at position $3$ (which was the original position $5$ which yields $\Delta_3=(1,2,1)$ and $E_3=\{23,34,12,24,45,56\}$.
\item In the last step, we add the edges $14$ and $46$ since the inner $2$ in $\Delta_3$ was obtained from the original vertex $4$.
\end{itemize}
The obtained graph is shown in \Cref{fig: HVG_432615}.
\begin{figure}[h]
    \centering
    \includegraphics[width=7.0cm]{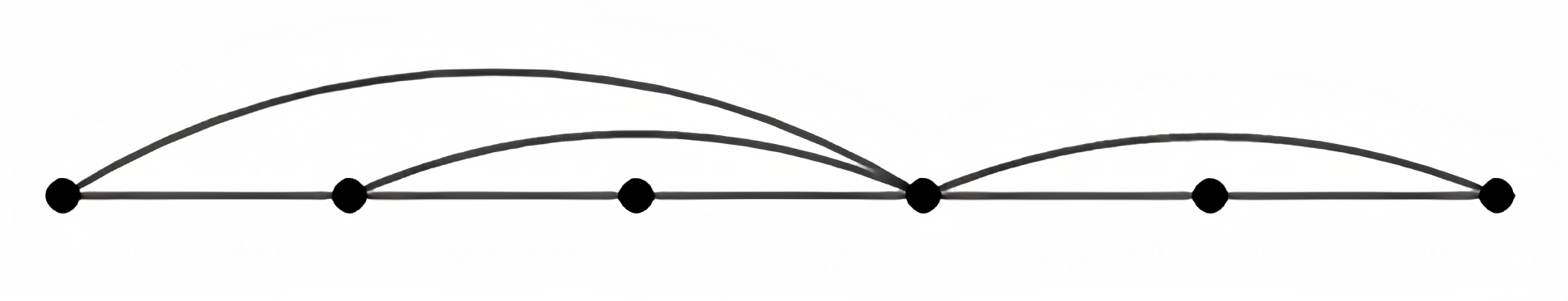}
    \caption{The unique HVG with ordered degree sequence $(2,3,2,5,2,2)$}
    \label{fig: HVG_432615}
\end{figure}
\end{exa}

Having \Cref{thm: degree sequence determines HVG} in mind it is natural to ask if this result can be generalized to arbitrary HVGs without restricting to data sequences with pairwise distinct entries. It can be verified computationally that this is possible for $N\leq 6$, i.e., any HVG in $\Gc_N$ is uniquely determined by its ordered degree sequence. However, already for $N=7$ this breaks down, since there exist $394$ HVGs compared to $391$ ordered  degree sequences. An instance of two HVGs with the same ordered degree sequence is shown in \Cref{figure: exam_same_degree}.
\begin{figure}[h]
	\centering
	\begin{subfigure}[b]{.45\textwidth}
        \centering
        \includegraphics[width=7.0cm]{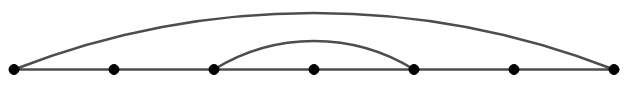}
        \caption{$\HVG((3,2,2,1,2,2,3))$}
        \label{fig: HVG_3221223}
	\end{subfigure}\hspace{5mm}%
	\begin{subfigure}[b]{.45\textwidth}
        \centering
        \includegraphics[width=7.0cm]{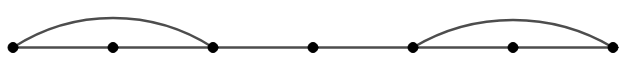}
        \caption{$\HVG((2,1,2,2,2,1,2))$}
        \label{fig: HVG_2122212}
	\end{subfigure}
	\caption{Two HVGs with the same ordered degree sequence $\Delta=(2,2,3,2,3,2,2)$.}
	\label{figure: exam_same_degree}
\end{figure}

We also want to point out that \Cref{thm: degree sequence determines HVG} cannot be generalized to unordered degree sequences since it is easy to construct examples of different HVGs with the same unordered degree sequence. For instance, the sequences $(4,1,2,3)$ and $(4,2,1,3)$ yield different HVGs having the same (unordered) degree sequence $(2,2,3,3)$ (see also \cite{Mansour}).

\subsection{Counting HVGs in $\Gc_{N,\neq}$~--~Catalan numbers}\label{sect: HVGs and Catalan}
The aim of this subsection is to prove \Cref{thm: Catalan numbers} (i). Namely, to show that HVGs in $\Gc_{N,\neq}$ are counted by the $(N-1)$-st Catalan number $C_{N-1}$ (see \cite{StanleyCatalan} for the numerous interpretations of these). 
We first introduce some notation. For $N,s\in \mathbb{N}$ with $2\leq s\leq N$, let 
$$
	\Gc_{N,\neq}^s=\left\{G\in \Gc_{N,\neq}~:~m_G(1)=s\right\}.
	$$
	Obviously, we have $|\Gc_{N,\neq}|=\sum_{s=2}^N|\Gc_{N,\neq}^s|$. We start by providing a relation between $|\Gc_{N,\neq}^s|$, $|\Gc_{s,\neq}^s|$ and $|\Gc_{N-s+1,\neq}|$.
	
\begin{lem}\label{lem: g_{n,s}}
	Let $N,s\in\N$ with $2\leq s\leq N$. Then 
	\begin{equation*}
|\Gc_{N,\neq}^s|= |\Gc_{s,\neq}^s|\cdot |\Gc_{N-s+1,\neq}|
	\end{equation*}
\end{lem}
\begin{proof}
Let $G\in \Gc_{N,\neq}^s$. It follows from \Cref{lem:easy properties} (iii) and the fact that $1s\in E(G)$ that $G_{[s]}\in \Gc_{s,\neq}^s$ and $G_{[s,N]}\in\Gc_{N-s+1,\neq}$. Since \Cref{lem:easy properties} (i), combined with $1s\in E(G)$, implies that $G$ does not have edges between vertices in $[s-1]$ and vertices in $[s+1,N]$, it holds that $G=G_{[s]}+G_{[s,N]}$, i.e., $G$ is uniquely determined by $G_{[s]}$ and $G_{[s,N]}$ and hence, $|\Gc_{N,\neq}^s|\leq|\Gc_{s,\neq}^s|\cdot |\Gc_{N-s+1,\neq}|$.

Conversely, let $G\in \Gc_{s,\neq}^s$, $H\in \Gc_{N-s+1,\neq}$. \Cref{lem:operations} (iii), together with $1s\in E(G)$ implies that $G+H\in \Gc_{N,\neq}^s$. Since $G=(G+H)_{[s]}$ and $H=(G+H)_{[s,N]}$, $G$ and $H$ are uniquely determined by $G+H$ and it follows that  $|\Gc_{N,\neq}^s|\geq|\Gc_{s,\neq}^s|\cdot |\Gc_{N-s+1,\neq}|$.
This finishes the proof.
\end{proof}

The next lemma will be crucial to count the graphs in $\Gc_{s,\neq}^s$.
	\begin{lem}\label{lem: connected to nonnested vertices}
	Let $G\in \Gc_{s,\neq}^s$. Then $N(s)=\Nc(G_{[s-1]})$.
	\end{lem}
	
	\begin{proof}
Let $D\in \mathbb{N}^s$ with $\HVG(D)=G$ and pairwise distinct entries. We first show that   $N(s)\subseteq\Nc(G_{[s-1]})$. To this end, let $\ell\in N(s)$, i.e., $\ell s\in E(G)$. Since in this case we must have $d_\ell>d_i<d_s$ for all $\ell<i<s$, there is no $uv\in E(G)$ with $1\leq u<\ell<v\leq s-1$. Hence, $\ell\in \Nc(G_{[s-1]})$.
For the reverse containment, consider $\ell\in \Nc(G_{[s-1]})$. The statement is trivially true for $\ell=1$ and $\ell=s-1$. For $\ell\in \Nc(G_{[s-1]})\setminus\{1,s-1\}$ assume by contradiction that $\ell s\notin E(G)$. Then there exists $\ell<j<s$ with $d_j>d_\ell$ or $d_j>d_s$. As $1s\in E(G)$, the latter case never occurs and therefore we must have $d_j>d_\ell$. In the following, we assume that $j$ is minimal with this property. Similarly, let $1\leq k<\ell$ maximal such that $d_k>d_\ell$. Note that such $k$ exists since $d_1>d_\ell$. It then follows that $kj\in E(G)$ and hence $kj\in E(G_{[s-1]})$ which implies $\ell\notin \Nc(G_{[s-1]})$, a contradiction. 
	\end{proof}
	
We want to remark that the same proof as above shows that \Cref{lem: connected to nonnested vertices} holds more generally for $G\in \Gc_{N,\neq}^s$. However, we do not need the statement in such generality. 
On the other hand, \Cref{lem: connected to nonnested vertices} does not generalize to arbitrary HVGs (see \Cref{exa: easy properties} for an example).
	
The next lemma is the last ingredient we need for the proof of \Cref{thm: Catalan numbers} (i).
	\begin{lem}\label{lem: g_{s,s}}
	Let $N\in\mathbb{N}$, $N\geq 2$. Then
	\begin{equation*}
	|\Gc_{N,\neq}^N|=|\Gc_{N-1,\neq}|.
	\end{equation*}
	\end{lem}
\begin{proof}
We show the claim by proving that 
$$
\Phi:\Gc_{N,\neq}^N\to \Gc_{N-1,\neq}: G\mapsto G_{[N-1]}$$ 
is a bijection. By \Cref{lem:easy properties} (iii) the map $\Phi$ is well-defined and it directly follows from \Cref{lem: connected to nonnested vertices} that $\Phi$ is injective. To show surjectivity, let $H\in \Gc_{N-1,\neq}$ and let $D=(d_1,\ldots,d_{N-1})\in\mathbb{N}^{N-1}$ be the standard sequence of $H$. Since all entries of $D$ are distinct, at most $N-1$ and $d_1=N-1$ it follows that $G=\HVG((d_1,\ldots,d_{N-1},N))\in \Gc_{N,\neq}^N$. Since clearly  $\Phi(G)=H$, we conclude that $\Phi$ is surjective. 
\end{proof}
Finally, we can provide the proof of \Cref{thm: Catalan numbers} (i).

\noindent{\sf Proof of \Cref{thm: Catalan numbers} (i)}
We show the claim via induction. For $N\in\{1,2\}$ there exists exactly one graph in $\Gc_{N,\neq}$ and since $C_0=C_1=1$, the claim is trivially true in this case.
Let $N\geq 2$. We then have
\begin{align*}
|\Gc_{N,\neq}|=&\sum_{s=2}^N |\Gc_{N,\neq}^s|=\sum_{s=2}^N |\Gc_{s,\neq}^s|\cdot |\Gc_{N-s+1,\neq}|\\
=&\sum_{s=2}^N|\Gc_{s-1,\neq}|\cdot |\Gc_{N-s+1,\neq}|\\
=& \sum_{s=2}^NC_{s-2}\cdot C_{N-s}=\sum_{s=0}^{N-2}C_{s}\cdot C_{N-s}=C_{N-1},
\end{align*}
where the second, third, fourth and sixth equality follow from \Cref{lem: g_{n,s}}, \Cref{lem: g_{s,s}}, the induction hypothesis and Segner's recurrence formula for the Catalan numbers, respectively. 
\qed

We end this subsection with an identity for the Catalan numbers, which we stumbled over in our study of HVGs but which we were unable to find in the literature. 

\begin{prop}\label{prop: Catalan identity}
Let $N\in \mathbb{N}$. Then
\begin{equation*}
C_{N}=1+\sum_{k=3}^{N+1}\frac{3}{2k-3}\binom{2k-3}{k}.
\end{equation*}
\end{prop}

\begin{proof}
We prove the statement via induction. Since $C_0=C_1=1$, the statement is trivially true for $N\in \{0,1\}$. Now let $N\geq 2$. In this case, we have
\begin{eqnarray*}
1+\sum_{k=3}^{N+1}\frac{3}{2k-3}\binom{2k-3}{k}&=&1+\sum_{k=3}^{N}\frac{3}{2k-3}\binom{2k-3}{k}+\frac{3}{2N-1}\binom{2N-1}{N+1}\\
&=&C_{N-1}+\frac{3}{2N-1}\binom{2N-1}{N+1}\\
&=&  \frac{(2N-2)!}{N (N-1)!(N-1)!}+\frac{3(2N-2)!}{(N+1)!(N-2)!}\\
&=& \frac{1}{N+1}\binom{2N}{N}=C_N
\end{eqnarray*}
where the second equality follows from the induction hypothesis and the fourth from an easy computation. 
\end{proof}

\subsection{HVGs and Parentheses}\label{sect: HVGs and Parentheses}
Since we have seen in \Cref{thm: Catalan numbers} (i) that HVGs in $\Gc_{N,\neq}$ are counted by the Catalan number $C_{N-1}$, it is natural to ask for a bijective proof of this statement. This is the goal of this subsection. More precisely, we provide an explicit bijection between $\Gc_{N,\neq}$ and the set $\Bc_{N-1}$ of balanced parantheses of length $N$, which are known to be counted by $C_{N-1}$ \cite[p.134 f.]{koshy2009catalan}. We use the definition for balanced parentheses from \cite[p. 155]{lehman2010mathematics}.
\begin{defi}\label{def: balanced_parentheses}
Let $\epsilon$ be the empty string. The set $\Bc$ of \emph{balanced parentheses} is recursively defined via
\begin{itemize}
    \item[(i)] $\epsilon\in \Bc$.  
    \item[(ii)] If $B_1, B_2\in \Bc$, then $[B_1]B_2\in\Bc.$
\end{itemize}
The set of balanced parentheses with $N$ pairs of parentheses is denoted by $\Bc_N$.
\end{defi}

It is easily seen from the definition that any balanced parentheses $B\in \Bc_N$ can be uniquely written in the form $B=[B_1]\ldots[B_k]$ with $B_j\in \Bc_{i_j}$ for $i_j\in\mathbb{N}$ and $\sum_{j=1}^ki_j=N-k$. We will refer to this representation as \emph{normal} representation of $B$ with blocks $B_1,\ldots,B_k$ and to $i_1,\ldots,i_k$ as the \emph{lengths} of the blocks.

We now state the main result of this section.

\begin{theorem}
    Let $N\in \mathbb{N}$, $N\geq 1$. Let $\psi_N:\Gc_{N,\neq}\to \Bc_{N-1}$ be recursively defined by $\psi_1(P_1)=\epsilon$ and 
    $$
    \psi_N(G)=[\psi_{i_2-i_1}(G_{[i_1+1,i_2]})]\cdots[\psi_{i_k-i_{k-1}}(G_{[i_{k-1}+1,i_k]})]
    $$
    if $N>1$ and $G\in \Gc_{N,\neq}$ with $\Nc(G)=\{i_1<\cdots < i_k\}$. Then, $\psi_N$ is a bijection.
\end{theorem}

\begin{proof}
Since $\sum_{j=2}^k(i_j-i_{j-1})=i_k-i_1=N-1$, it is easily seen by induction on $N$ that the map $\psi_N$ is well-defined. Moreover, $\psi_N(G)$ is given in normal representation.

As it follows from \Cref{thm: Catalan numbers} (i) and \cite[p.134 f.]{koshy2009catalan} that $|\Gc_{N,\neq}|=|\Bc_{N-1}|=C_{N-1}$, it suffices to show that $\psi_N$ is injective for every $N\geq 1$. For $N=1$, this is trivially true since $\Gc_{1,\neq}=\{P_1\}$. Let $N\geq 2$ and let $G,H\in \Gc_N$ such that $G\neq H$. If $\Nc(G)\neq \Nc(H)$, then $\psi_N(G)$ and $\psi_N(H)$ must have blocks of different lengths, which already implies $\psi_N(G)\neq \psi_N(H)$. Assume that $\Nc(G)=\Nc(H)=\{i_1<\cdots <i_k\}$. As $G\neq H$, there exists $2\leq j\leq k$ with $G_{[i_{j-1}+1,i_j]}\neq H_{[i_{j-1},i_j]}$. The induction hypothesis implies that $\psi_{i_j-i_{j-1}}(G_{[i_{j-1}+1,i_j]})\neq \psi_{i_j-i_{j-1}}(H_{[i_{j-1}+1,i_j]})$ and hence  $\psi_N(G)\neq \psi_N(H)$.
\end{proof}

The next example illustrates the bijection $\psi_N$.

\begin{exa}\label{exa: psi_n}
For the graph $G$ in \Cref{figure: exa_psi_n} we get $\psi_{10}(G)=[[[\ ]\ [\ ]\ [\ ]]\ [\ ]]\ [[\ ]\ [\ ]]\in\Bc_9.$ The process of how $\psi$ works is visualized in \Cref{figure: exa_psi_n_2}.
\begin{figure}[h]
	\centering
	\includegraphics[width=.5\linewidth]{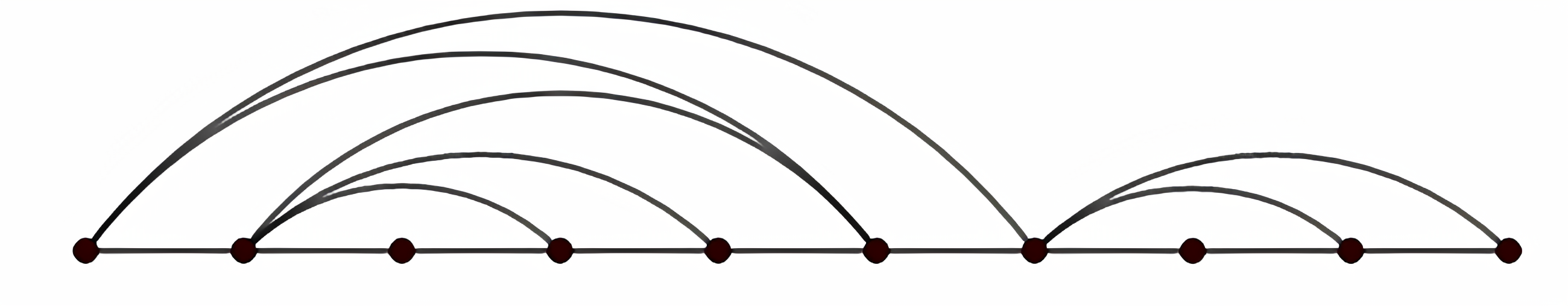}
	\caption{$G=\HVG(D)$ with $D=(10,6,2,4,5,8,9,1,3,7)$.}
	\label{figure: exa_psi_n}
\end{figure}
\begin{figure}[h]
	\centering
	\includegraphics[width=.6\linewidth]{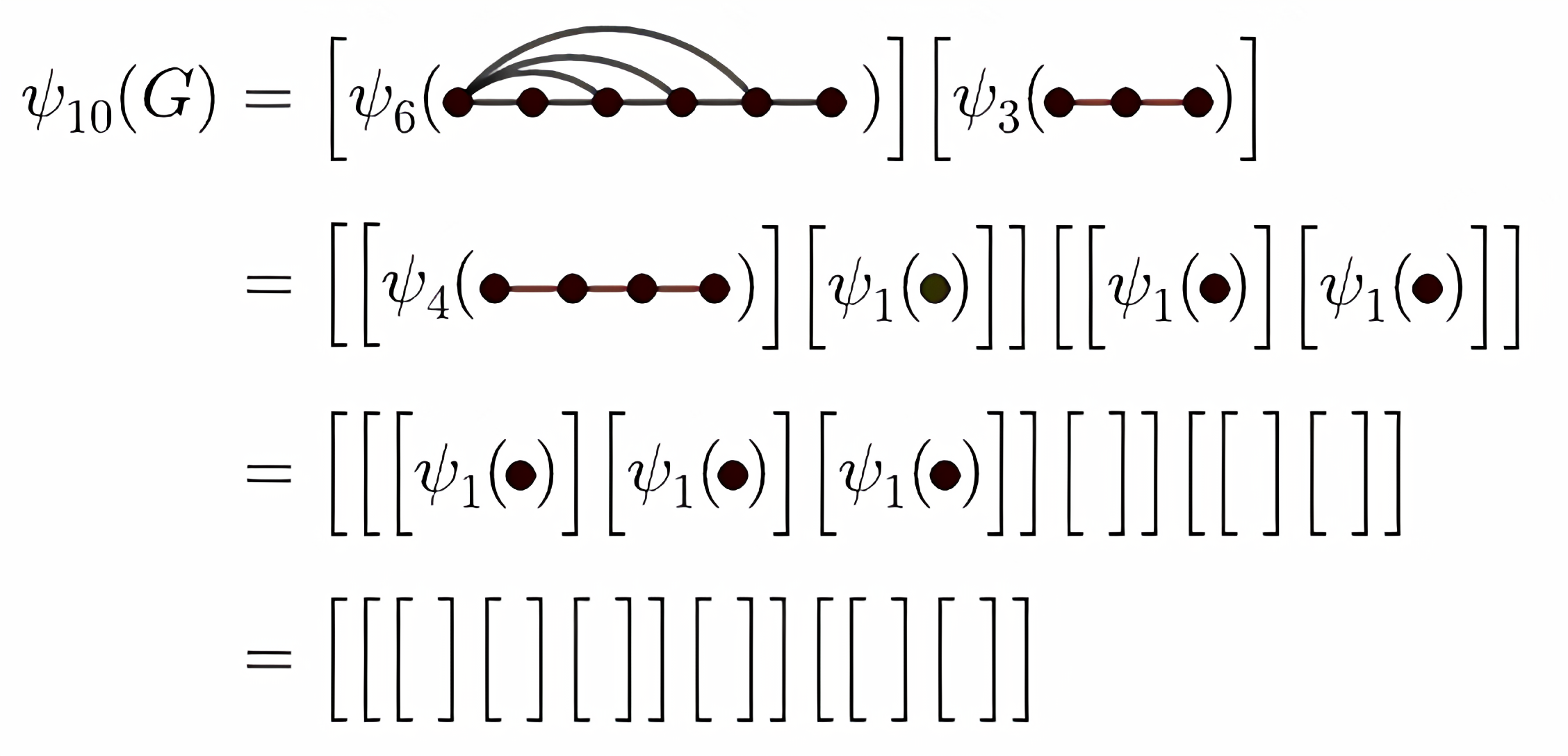}
	\caption{Applying $\psi_{10}$ to $G$}
	\label{figure: exa_psi_n_2}
\end{figure}
\end{exa}

\begin{rem}
We want to remark that it is easily seen that the inverse map $\psi^{-1}_N:\Bc_{N-1}\to\Gc_{N,\neq}$ of $\psi_N$ is given by
$\psi^{-1}_1(\epsilon)=P_1$ and 
$$
\psi_N^{-1}(B)= \overline{\psi^{-1}_{i_1}(B_1)}+\cdots+\overline{\psi^{-1}_{i_k}(B_k)},
$$
if $N>1$ and $B=[B_1]\cdots[B_k]\in\Bc_{N-1}$ with $B_j\in \Bc_{i_j-1}$ and $\sum_{j=1}^k{i_j}=N-1$.
Here, for an HVG $G$, we denote by $\overline{G}$ the HVG $$((\{1,2\},\{12\})+G)\cup \big\{1(i+1)~:~ i\in \Nc(G)\big\},$$
i.e., $\overline{G}$ is obtained from $G$ by adding a ``new'' vertex $1$ that is connected to all non-nested vertices of $G$.
\end{rem}

\section{Horizontal Visibility Graphs from arbitrary data} \label{sec: HVG arbitrary data}
While in the previous section we were focusing on HVGs corresponding to data sequences without equal entries, we will now omit this restriction and allow arbitrary data sequences. 
As before, it follows from \Cref{rem:integral} that we only need to consider integral data sequences.

Our first goal is to describe an explicit method to construct a data sequence $D$ that realizes a given $G\in \Gc_N$ as its HVG. This is very similar to \Cref{thm:dataGraph}. In the second part of this section, we turn to a more combinatorial problem: Namely, counting HVGs in $\Gc_N$. In particular, we prove \Cref{thm: Catalan numbers} (ii).

\subsection{From HVGs to data sequences}\label{sect: Integer vector algo}
In the following, we are asking the analogous question to the one posed in  \Cref{subsect:GraphToData}. Namely, given $N\in\mathbb{N}$, $G=([N],E)\in\Gc_N$ we are searching a data sequence $D\in\mathbb{N}^N$ realizing $G$. An answer is provided by the next theorem, which uses the same notations as in \Cref{subsect:GraphToData}.

\begin{theorem}\label{thm: sequence arbitrary data}
Let $N\in\mathbb{N}$ and $G\in \Gc_N$. For $1\leq i\leq N$ let
$$
d_i=N-d_{\nest}(i).
$$
Then $D=(d_1,\ldots,d_N)$ realizes $G$, i.e., $\HVG(D)=G$.
\end{theorem}

\begin{proof}
Let $\widetilde{D}=(\widetilde{d}_1,\ldots,\widetilde{d}_N)\in\mathbb{N}^N$ with $\HVG(\widetilde{D})=G$ and let $H=\HVG(D)$. Verbatim the same arguments as in the proof of \Cref{thm:dataGraph} show that $E(G)\subseteq E(H)$. 

For the reverse containment, let  $ij\in E(H)$ with $i+1<j$ and assume that there exists $i<k<j$ with $\widetilde{d}_k\geq \min(\widetilde{d}_i,\widetilde{d}_j)$. Since in contrast to the proof of \Cref{{thm:dataGraph}} everything is symmetric with respect to $i$ and $j$ one can assume that  $\min(\widetilde{d}_i,\widetilde{d}_j)=\widetilde{d}_i$. As in  Case 1 of the proof of \Cref{thm:dataGraph} it follows that $d_{\nest}(i)\geq d_{\nest}(k)$ which directly implies $d_i\leq d_k$, yielding a contradiction. 
Since $i(i+1)$ for $1\leq i\leq N-1$ lies in any HVG, we conclude $G=H$.
\end{proof}
The graph in \Cref{figure: exa_easy properties 1} can be represented with \Cref{thm: sequence arbitrary data} via $D=(4,3,3,4)$.

\subsection{Counting HVGs~--~Schr\"oder numbers}
The aim of this section is to prove \Cref{thm: Catalan numbers} (ii). Namely, to show that the number of HVGs of length $N$ is given by the $(N-2)$-nd large Schr\"oder number $r_{N-2}$. Those are known to count several combinatorial objects including certain types of lattice paths (see \cite{ShapiroSulanke}). We start by providing relevant definitions. 
\begin{defi}
 A \emph{bracketing} $B$ of a string of identical letters $x$ is
 \begin{itemize}
     \item either a single letter $x$, or
     \item $B=(B_1,\ldots,B_k)$, where $k\geq 2$, and $B_1,\ldots,B_k$ are bracketings and brackets around a single letter as well as the outer surrounding brackets are omitted.
 \end{itemize}
 The bracketing $x\cdots x$ without any brackets will be referred to as a \emph{trivial bracketing}. The \emph{length} $\ell(B)$ of a bracketing $B$  is defined to be the number of enclosed letters and we use $\Bs_N$ to denote the set of bracketings of length $N$.
\end{defi}
It is easy to see from the definition that every bracketing $B\in\Bs_{N}$ has a unique representation of the form $B=B_1\cdots B_k$, where for $1\leq i\leq k$, $B_i$ is either a trivial bracketing, or, $B_i=(\tilde{B}_i)$ for a bracketing $\tilde{B}_i$ and no two trivial bracketings are adjacent. The last condition means that adjacent trivial bracketings are grouped together into a trivial bracketing of maximal length. We call this representation the \emph{normal form} of a bracketing. 
$s_{N}=|\Bs_{N+1}|$ is called the \emph{$N$-th little Schr\"oder number} \cite{schroeder1870Probleme}. 
It is well-known that $r_N=2s_N$. 
 Similar to \Cref{sect: HVGs and Catalan} we write $\Gc_N^s$ for the set of HVGs $G$ in $\Gc_N$ with $m_G(1)=s$.  The next lemma allows us to reduce the proof of \Cref{thm: Catalan numbers} (ii) to counting HVGs without $1N$.

\begin{lem}\label{lem: G_1n G_setminus 1n}
	For $N\geq 3$, we have
	\[
	|\Gc_{N}^N|=|\Gc_N\setminus\Gc_{N}^{N}|, \qquad \text{i.e.},\qquad 	|\Gc_{N}|=2	|\Gc_{N}^N|.
	\]
\end{lem}
\begin{proof}
It is easy to see that the map
	\[
	\phi:\Gc_{N}^N\to \Gc_N\setminus \Gc^N_{N}: G\mapsto G\setminus \{1N\}
	\]
	is a bjection. Indeed, it follows from \Cref{lem:operations} (i) and (ii) that $\phi$ is well-defined and surjective, respectively. Since the injectivity is obvious, the claim follows.
\end{proof}

As $r_N=2s_N$, the next statement completes the proof of \Cref{thm: Catalan numbers} (ii).

\begin{theorem}\label{thm:Schroeder}
    Let $N\in\mathbb{N}$, $N\geq 2$. Then
    $$
|\Gc_N^N|=s_{N-2}.
    $$
\end{theorem}

\begin{proof}
We clearly have $|\Gc_2^2|=1=s_0$ and hence the claim holds in this case. 

For ease of notation, we set $\Gc_N^*=\Gc_N\setminus \Gc_N^N$. To show the claim we provide a bijection between  $\Bs_{N}$ and $\Gc_{N+1}^*$ for $N\geq 2$. 
We consider the map $ \xi_{N}:\Bs_{N} \to \Gc_{N+1}^*$ which is defined by $\xi_2(xx)=P_3$, $\xi_N(x\cdots x)=P_{N+1}$ for any $N\geq 2$. If $N\geq 3$ and  $B=B_1\cdots B_k\in \Bs_N$ is in normal form with non-trivial blocks $B_{i_1},\ldots,B_{i_r}$, where $i_1<i_2<\cdots <i_r$, we recursively define
$$
\xi_N(B)=\xi_{\ell(B_1)}(\widetilde{B}_1)+\cdots + \xi_{\ell(B_k)}(\widetilde{B}_k) \cup\{(\sum_{j=1}^{i_m-1}\ell(B_j)+1)(\sum_{j=1}^{i_m}\ell(B_j)+1) ~:~1\leq m\leq r\},
$$
where $B_j=\widetilde{B}_{j}\in B_{\ell(B_j)}$ if $B_j$ is trivial and $B_j=(\widetilde{B}_j)$, otherwise.   We also set $\xi_1(x)=P_2$. 
As $\sum_{i=1}^k(\ell(B_i)+1)-(k-1)=\sum_{i=1}^k\ell(B_i)+1=N+1$ and $\widetilde{B}_j\in \Bs_{\ell(B_j)}$, it follows by induction on $N$ and  \Cref{lem:operations} (iii) that $\xi_N$ is well-defined. 
The map $\xi_N$ is obviously injective for $N=2$, and for $N\geq 3$, using induction, we get injectivity directly from the definition of $\xi_N$. It remains to show that $\xi_N$ is surjective. For $N=2$, this is clear. Assume $N\geq 3$ and let $G\in \Gc_{N+1}^*\setminus \{P_{N+1}\}$. Since $1(N+1)\notin E(G)$, there exists a non-nested vertex $s$ of $G$ with $1<s<N+1$. Choosing $s$ maximal, it follows that $s(N+1)\in E(G)$. By \Cref{lem:easy properties} (iii) and \Cref{lem:operations} (i) it holds that $G_{[s,N+1]}\setminus \{s(N+1)\}\in\Gc_{N+2-s}^*$. We now distinguish two cases. If $1s\notin E(G)$, then again by  \Cref{lem:easy properties} (iii) we have $G_{[1,s]}\in \Gc_{s}^*$. By induction, there exist $B_1\in\Bs_{s-1}$ and $B_2\in \Bs_{N-s+1}$ such that $\xi_{s-1}(B_1)=G_{[1,s]}$ and $\xi_{N-s+1}(B_2)=G_{[s,N+1]}\setminus \{s(N+1)\}$. As $B_1(B_2)\in \Bs_N$, we further conclude that
\begin{eqnarray*}
\xi_N(B_1(B_2))=&\xi_{s-1}(B_1)+\xi_{N-s+1}(B_2)\cup\{s(N+1)\}\\
=&G_{[1,s]}+G_{[s,N+1]}\setminus\{s(N+1)\}\cup\{s(N+1)\}=G.
\end{eqnarray*}
If $1s\in E(G)$, then $G_{[1,s]}\setminus\{1s\}\in \Gc_s^*$ and there exists $B_1\in \Bs_{s-1}$ with $\xi_{s-1}(B_1)=G_{[1,s]}\setminus\{1s\}$. A similar computation as in the previous case shows that $\xi_N((B_1)(B_2))=G$. Hence, the map $\xi_N$ is surjective.  This finishes the proof.
\end{proof}

We provide an example to illustrate the bijection $\xi_N$.
\begin{exa}
Applying $\xi_8$ to 
	\[
	B=(xx)\big((xxx)x(xx)\big)
	\]
	results in
	\begin{eqnarray*}
		\xi_{8}(B)&=&\xi_2(xx)+\xi_6((xxx)x(xx))\cup\{13,39\}\\
	&=&\xi_2(xx)+\xi_3(xxx)+\xi_1(x)+\xi_2(xx)\cup\{13,39\}\cup\{36,79\},
	\end{eqnarray*}
	and the graph obtained is shown in \Cref{img: HVG Thm Exa phi}.

\begin{figure}[h]
	\includegraphics[width=10.0cm]{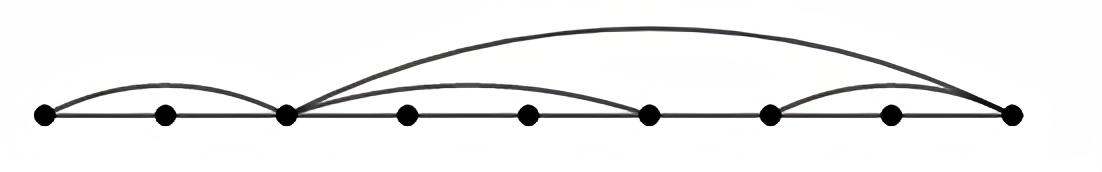}
	\centering
	\caption{$\xi_{8}(B)$.}
	\label{img: HVG Thm Exa phi}
\end{figure}
\end{exa}

The little Schr\"oder number $s_{N-2}$ is known to count a variety of combinatorial objects, including \emph{dissections} of a convex polygon $\Pi_N$ on $N$ vertices, labeled $1,\ldots,N$. Here, a dissection of $\Pi_N$ is defined as a subdivision of $\Pi_N$ into polygonal regions via non-crossing diagonals between vertices of $\Pi_N$ (see \cite[Section 3]{Flajolet1999Dissections}). In other words, a dissection is a non-crossing graph containing the cycle $1,\ldots,N,1$. In particular, any HVG in $\Gc_N$ with $1N\in E(G)$ can naturally be viewed as a dissection. \Cref{thm:Schroeder} even implies that every dissection can be obtained this way.

\begin{cor} \label{cor: dissections}
For $N\geq 3$, the sets $\Gc_N^N$ and $\Pi_N$ are in natural bijective correspondence, where the map is given by the identity.
\end{cor}

\section{Open Problems and future work} \label{sec: outlook}
The main goals of this article lay on the reconstruction of HVGs in $\Gc_N$ from a given ordered degree sequence and in counting HVGs in $\Gc_N$ and $\Gc_{N,\neq}$ which led us to objects that are counted by the large Schr\"oder and Catalan numbers, respectively. From our results several open questions arose that we now briefly discuss.

As an extension of HVGs it is natural to consider the more general class of  \emph{visibility graphs} (VGs for short) \cite{lacasa2008time}, defined as follows. Given a data sequence $(t_1,d_1),\ldots,(t_N,d_N)$, where the $t_i$ are time points, the visibility graph of this sequence is the graph on vertex set $[N]$, where $ij$ is an edge iff $d_k<d_j+(d_i-d_j)\frac{t_j-t_k}{t_j-t_i}$ for all $t_k$ with $t_i<t_k<t_j$. It is immediately seen that this graph always contains the HVG of the data sequence $(d_1,\ldots,d_N)$ as a subgraph. In line with \Cref{thm: Catalan numbers} it is natural to ask for the cardinality of VGs on a fixed number of vertices. To this end, in a first step, we successively constructed VGs from random data-sequences of length up to $7$ until no new VGs were found. Though there is no guarantee to have exhausted the whole set of VGs on up to $7$ nodes in this way, we suspect that the number of those VGs are the ones displayed in the next table:\\
\begin{center}
   $\begin{array}{c|c}
    N & \text{number of VGs on }N \text{ nodes} \\ \hline
     1 & 1 \\
     2 & 1 \\
     3 & 2 \\
     4 & 6 \\
     5 & 25 \\
     6 &  138\\
     7 &  972\\
     8 & 8477
\end{array}$. 
\end{center}
This sequence seems to be sequence A007815 in OEIS \cite{OEIS}, which counts so-called persistent graphs on $N$ nodes. On the one hand, every VG is a persistent graph. On the other hand, there exist persistent graphs which are not VGs \cite{ameer2020terrain}. In particular, sequence A007815 is just an upper bound for the cardinality in question. So, we do not even have a conjectured answer to the following question.

\begin{question}\label{qu:numbers}
What is the number of VGs on $N$ nodes?
\end{question}

Since the set of HVGs on $N$ nodes is contained in the set of VGs on $N$ nodes, one possible way to answer \Cref{qu:numbers} is by means of the following question:

\begin{question}
Can one characterize (graph-theoretically) the VGs that are not HVGs?
\end{question}

Moreover, one could asked under what constrains on a given data-sequence the associated VG is actually an HVG.  More precisely, it would be interesting to consider the following problem:
\begin{question}
Can one characterize data sequences such that the corresponding VG is an HVG? If so, is it possible to construct a data sequence having the considered VG as its HVG? Does the same data sequence work?
\end{question}

Motivated by what is happening for HVGs, the next question arises.
\begin{question}
Is there a difference between VG associated to sequences with pairwise distinct entries (when restricting to the second coordinate) in contrast to VGs associated to arbitrary sequences where equal entries in the second coordinate are allowed?
\end{question}

In this article we have not touched the run time of algorithms for the construction of HVGs from a given data sequence. Several such algorithms exist \cite{lan2015fast, Yela2020HVGCoding, luque2009horizontal, lacasa2012time}. The fastest known algorithm which is due to Lacasa et al. claims to have a run time of  $O(N)$ for special classes of sequences of $N$ data points. In ongoing work, the last two authors constructed an algorithm which has a run time in $O(N)$ for general time series \cite{KoehneSchmidt2021}.   \Cref{fig: runtime_comparison_bst_lta} shows a comparison of the computation time of this proposed algorithm, the binary search tree (BST) approach in \cite{Yela2020HVGCoding}, and Lacasa's algorithm  for random walks of length up to $10^5$.

\begin{figure}[h]
    \centering
    \includegraphics[width=10.0cm]{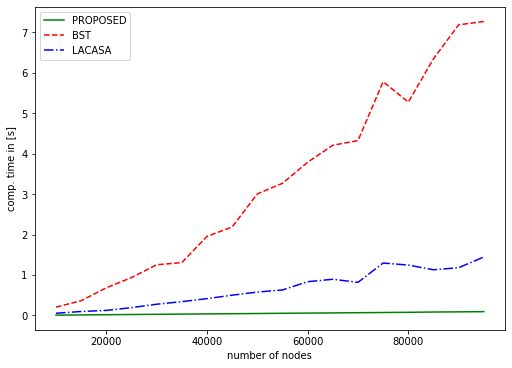}
    \caption{Comparison of different HVG algorithms}
    \label{fig: runtime_comparison_bst_lta}
\end{figure}
~\\



\bibliographystyle{alpha}
\bibliography{references}

\end{document}